\documentclass[12pt,reqno]{amsart}

\usepackage[cp1251]{inputenc}
\usepackage[T2A]{fontenc}
\usepackage[english,russian]{babel}
\usepackage{amsmath, amsthm, amscd, amsfonts, amssymb, graphicx, color}
\usepackage{geometry}
\usepackage{cite}
\theoremstyle{plain}
\newtheorem{theorem}{Теорема}[section]
\newtheorem{lemma}{Лемма}[section]
\newtheorem{proposition}{Предложение}[section]

\theoremstyle{definition}

\numberwithin{equation}{section}

\begin{document}

\begin{flushleft}


\textbf{Tetiana Zinchenko and Aleksandr Murach}\\
\small(Institute of Mathematics, National Academy of Sciences of Ukraine, Kyiv)

\medskip

\large\textbf{PETROVSKII ELLIPTIC SYSTEMS IN THE EXTENDED SOBOLEV SCALE}\normalsize

\medskip\medskip

\textbf{Татьяна Зинченко и Александр Мурач}\\
\small(Институт математики НАН Украины, Киев)

\medskip

\large\textbf{ЭЛЛИПТИЧЕСКИЕ ПО ПЕТРОВСКОМУ СИСТЕМЫ В РАСШИРЕННОЙ СОБОЛЕВСКОЙ ШКАЛЕ}\normalsize

\end{flushleft}

\medskip

\noindent Petrovskii elliptic systems of linear differential equations given on a closed smooth manifold are investigated on the extended Sobolev scale. This scale consists of all Hilbert spaces that are interpolation spaces with respect to the Hilbert Sobolev scale. Theorems on the sol\-va\-bi\-lity of the elliptic systems on the extended Sobolev scale are proved. An \textit{a priori} estimate for solutions is obtained,
and their regularity is studied.

\medskip

\noindent Эллиптические по Петровскому системы линейных ди\-ф\-ференциальных уравнений, заданные на замкнутом гладком многообразии, исследованы в расширенной соболевской шкале. Она состоит из всех гильбертовых пространств, интерполяционных относительно
гильбертовой соболевской шкалы. Доказаны теоремы о разрешимости эллиптических систем в расширенной соболевской шкале. Получена априорная оценка решений и исследована их локальная регулярность.

\bigskip



\section{Введение}\label{Sec1}
Хорошо известна фундаментальная роль пространств Соболева в теории уравнений с
частными производными, особенно в теории эллиптических дифференциальных уравнений.
Линейные эллиптические операторы обладают рядом характерных свойств в шкалах
соболевских пространств: фредгольмовость (т.~е. конечность индекса уравнения),
априорные оценки решений, локальное повышение регулярности решений. Эти свойства
имеют важные приложения, причем наиболее содержательные результаты получаются в
гильбертовом случае (см., например, обзоры \cite{Agranovich94, Agranovich97}).

В этой связи несомненный интерес представляют гильбертовы функциональные
пространства, интерполяционные относительно ги\-льбертовой соболевской шкалы.
Поскольку при интерполяции наследуется ограниченность линейных операторов, а также
их фредгольмовость, то шкалы этих пространств служат эффективным инструментом для
исследования эллиптических операторов.

Класс \emph{всех} гильбертовых пространств, интерполяционных относительно
гильбертовой соболевской шкалы был выделен и изучен в \cite{09Dop3, 13UMJ3} и
\cite[п.~2.4]{MikhailetsMurach10}. При этом были использованы фундаментальные
теоремы Ж.~Питре \cite{Peetre66, Peetre68} и В.~И.~Овчинникова
\cite[п.~11.4]{Ovchinnikov84} теории интерполяции пространств. Этот класс образован
пространствами Хермандера \cite[п. 2.2]{Hermander63}, для которых показателем
гладкости служит произвольная радиальная функция, RO-меняющаяся на $+\infty$. Такой
класс естественно именовать расширенной соболевской шкалой (посредством
интерполяционных пространств). В упомянутых работах \cite{09Dop3,
MikhailetsMurach10} найдены ее применения к скалярным эллиптическим операторам.

В настоящей статье исследованы эллиптические по Петровскому
\cite[с.~235]{Petrovskii86} системы линейных дифференциальных уравнений в
расширенной соболевской шкале, заданной на гладком замкнутом многообразии.
Установлены теоремы об ограниченности и фредгольмовости оператора, порожденного
системой. Получены новые априорные оценки решений и исследована локальная
регулярность решений в расширенной соболевской шкале.

В этой шкале равномерно эллиптические на $\mathbb{R}^{n}$ системы дифференциальных
уравнений были изучены в \cite{09UMJ3, 12UMJ11}, а эллиптические с параметром
операторы~--- в \cite{13MFAT1}. Для более узкого класса пространств Хермандера
(уточненная соболевская шкала) теория общих эллиптических уравнений и эллиптических
краевых задач построена В.~А.~Михайлецом и вторым автором в серии работ, среди
которых укажем статьи [14--20], обзор \cite{12BJMA2} и монографию
\cite{MikhailetsMurach10}.

Отметим, что в последнее время пространства Хермандера и их различные аналоги,
именуемые пространствами обобщенной гладкости, вызывают значительный интерес как
сами по себе, так и с точки зрения приложений [22--25].

\section{Постановка задачи}\label{Sec2}
Пусть $\Gamma$~--- бесконечно гладкое замкнутое (т.~е. компактное и без края)
многообразие размерности $n\geq1$. Предполагается, что на $\Gamma$ задана некоторая
$C^{\infty}$-плотность $dx$. Линейные топологические пространства
$C^{\infty}(\Gamma)$ основных функций и $\mathcal{D}'(\Gamma)$ обобщенных функций
(распределений), заданных на $\Gamma$, рассматриваются как взаимно
ан\-тидвойственные относительно расширения по непрерывности скалярного произведения
в $L_{2}(\Gamma,dx)$. Это расширение обозначаем через $(h,\omega)_{\Gamma}$, где
$h\in\mathcal{D}'(\Gamma)$ и $\omega\in C^{\infty}(\Gamma)$. Мы рассматриваем
комплекснозначные функции и распределения, причем последние предполагаются
антилинейными функционалами.

Пусть на многообразии $\Gamma$ задана система $p\geq2$ линейных дифференциальных
уравнений
\begin{equation}\label{f2.1}
\sum _{k=1}^p A_{j,k}u_{k}=f_{j},\quad j=1,\ldots,p.
\end{equation}
Здесь $A_{j,k}$, где $j,k=1,\ldots,p$,~---  скалярные линейные дифференциальные
операторы на $\Gamma$ произвольных порядков c коэффициентами класса
$C^{\infty}(\Gamma)$. Для каждого номера $k\in\{1,\ldots,p\}$ положим
$$
m_k:=\max\{\mathrm{ord}\,A_{1,k},\ldots,\mathrm{ord}\,A_{p,k}\}.
$$
Уравнения \eqref{f2.1} понимаются в смысле теории распределений на $\Gamma$.

Запишем систему \eqref{f2.1} в матричной форме $Au=f$, где
$A:=(A_{j,k})_{j,k=1}^{p}$~--- матричный дифференциальный оператор, а
$$
u:=\mathrm{col}(u_{1},\ldots,u_{p}),\quad f:=\mathrm{col}(f_{1},\ldots,f_{p})
$$
--- функциональные столбцы.

Рассмотрим главный символ оператора~--- матрицу
$$
A^{(0)}(x,\xi):=\bigl(A_{j,k}^{(0)}(x,\xi)\bigr)_{j,k=1}^p,\quad x\in\Gamma,\;\;
\xi\in T_{x}^{\ast}\Gamma.
$$
Здесь $A_{j,k}^{(0)}(x,\xi)$ --- главный символ дифференциального оператора
$A_{j,k}$ в случае, когда $\mathrm{ord}\,A_{j,k}=m_k$, либо
$A_{j,k}^{(0)}(x,\xi)\equiv0$ в случае, когда $\mathrm{ord}\,A_{j,k}<m_k$. Как
обычно, через $T_{x}^{\ast}\Gamma$ обозначено кокасательное пространство к
многообразию $\Gamma$ в точке $x\in\Gamma$.

Предполагается, что система \eqref{f2.1} эллиптическая по Петровскому на
многообразии $\Gamma$, т.е.
\begin{equation*}
\det A^{(0)}(x, \xi)\neq0\quad\mbox{для всех}\quad x\in\Gamma,\;\;\xi\in
T_{x}^{\ast}\Gamma\setminus\{0\}
\end{equation*}
(см., например, \cite[п. 3.2 b]{Agranovich94}).

Введем функциональные пространства, в которых исследуется система \eqref{f2.1}. Они
параметризованы функциональным параметром $\varphi\in\mathrm{RO}$, где RO~---
множество всех измеримых по Борелю функций
$\varphi:[1,\infty)\rightarrow(0,\infty)$, для которых существуют числа $a>1$ и
$c\geq1$ такие, что
$$
c^{-1}\leq\frac{\varphi(\lambda t)}{\varphi(t)}\leq c\quad\mbox{для любых}\quad
t\geq1,\;\lambda\in[1,a]
$$
(вообще говоря, постоянные $a$ и $c$ зависят от $\varphi\in\mathrm{RO}$). Такие
функции называют RO (или OR)-меняющимися на бесконечности. Класс RO-меняющихся
функций введен В.~Г.~Авакумовичем в 1936 г. и достаточно полно изучен (см.
\cite{BinghamGoldieTeugels89} (п. 2.0~-- 2.2) и \cite{Seneta76} (приложение~1)).

Пусть $\varphi\in\mathrm{RO}$. Определим необходимые нам пространства сначала на
$\mathbb{R}^{n}$, а затем на многообразии $\Gamma$; при этом мы следуем \cite[п.
2.4]{MikhailetsMurach10} (см. также \cite{09Dop3}).

Линейное пространство $H^{\varphi}(\mathbb{R}^{n})$ состоит, по определению, из всех
распределений $w\in\mathcal{S}'(\mathbb{R}^{n})$ таких, что их преобразование Фурье
$\widehat{w}:=\mathcal{F}w$ локально суммируемо по Лебегу в $\mathbb{R}^{n}$ и
удовлетворяет условию
$$
\int\limits_{\mathbb{R}^{n}}\varphi^2(\langle\xi\rangle)\,|\widehat{w}(\xi)|^2\,d\xi
<\infty.
$$
Здесь, как обычно, $\mathcal{S}'(\mathbb{R}^{n})$~--- линейное топологическое
пространство Л.~Шварца медленно растущих распределений, заданных в $\mathbb{R}^{n}$,
а $\langle\xi\rangle:=(1+|\xi|^{2})^{1/2}$~--- сглаженный модуль вектора
$\xi\in\mathbb{R}^{n}$. В пространстве $H^{\varphi}(\mathbb{R}^{n})$ определено
скалярное произведение распределений $w_1$, $w_2$ по формуле
$$
(w_1,w_2)_{H^{\varphi}(\mathbb{R}^{n})}:=
\int\limits_{\mathbb{R}^{n}}\varphi^2(\langle\xi\rangle)\,
\widehat{w_1}(\xi)\,\overline{\widehat{w_2}(\xi)}\,d\xi.
$$
Оно задает на $H^{\varphi}(\mathbb{R}^{n})$ структуру гильбертова пространства и
определяет норму
$\|w\|_{H^{\varphi}(\mathbb{R}^{n})}:=(w,w)_{H^{\varphi}(\mathbb{R}^{n})}^{1/2}$.

Пространство $H^{\varphi}(\mathbb{R}^{n})$~--- гильбертов изотропный случай
прост\-ранств $B_{p,k}$, введенных и систематически исследованных Л.~Хермандером
\cite[п. 2.2]{Hermander63} (см. также \cite[п. 10.1]{Hermander83}). Именно,
$H^{\varphi}(\mathbb{R}^{n})=B_{p,k}$, если $p=2$ и
$k(\xi)=\varphi(\langle\xi\rangle)$ при $\xi\in\mathbb{R}^{n}$. Отметим, что при
$p=2$ пространства Хермандера совпадают с пространствами, введенными и изученными
Л.~Р.~Волевичем и Б.~П.~Панеяхом \cite[\S~2]{VolevichPaneah65}.

Определим теперь аналог пространства $H^{\varphi}(\mathbb{R}^{n})$ для многообразия
$\Gamma$. Выберем какой-либо конечный атлас из $C^{\infty}$-структуры на $\Gamma$,
образованный локальными картами
$\alpha_{j}:\nobreak\mathbb{R}^{n}\leftrightarrow\nobreak U_{j}$, $j=1,\ldots,q$,
где открытые множества $U_{j}$ составляют конечное покрытие многообразия $\Gamma$.
Пусть функции $\chi_{j}\in C^{\infty}(\Gamma)$, $j=1,\ldots,q$, образуют разбиение
единицы на $\Gamma$, удовлетворяющее условию $\mathrm{supp}\,\chi_{j}\subset\nobreak
U_{j}$.

Линейное пространство $H^{\varphi}(\Gamma)$ состоит, по определению, из всех
распределений $h\in\nobreak\mathcal{D}'(\Gamma)$ таких, что
$(\chi_{j}h)\circ\alpha_{j}\in H^{\varphi}(\mathbb{R}^{n})$ для каждого
$j=1,\ldots,q$. Здесь $(\chi_{j}h)\circ\alpha_{j}$~--- представление распределения
$\chi_{j}h$ в локальной карте $\alpha_{j}$. В пространстве $H^{\varphi}(\Gamma)$
определено скалярное произведение распределений $h_{1}$ и $h_{2}$  по формуле
$$
(h_{1},h_{2})_{\varphi}:=\sum_{j=1}^{q}\,((\chi_{j}h_{1})\circ\alpha_{j},
(\chi_{j}h_{2})\circ\alpha_{j})_{H^{\varphi}(\mathbb{R}^{n})}.
$$
Оно задает на $H^{\varphi}(\Gamma)$ структуру гильбертова пространства и определяет
норму $\|h\|_{\varphi}:=(h,h)_{\varphi}^{1/2}$.

Пространство $H^{\varphi}(\Gamma)$ не зависит (с точностью до эквивалентности норм)
от выбора локальных карт и разбиения единицы на~$\Gamma$ \cite[теорема
2.21]{MikhailetsMurach10}. Это пространство сепарабельно; справедливы непрерывные и
плотные вложения $C^{\infty}(\Gamma)\hookrightarrow
H^{\varphi}(\Gamma)\hookrightarrow\mathcal{D}'(\Gamma)$.

Если $\varphi(t)=t^{s}$ для всех $t\geq1$ при некотором $s\in\mathbb{R}$, то
$H^{\varphi}(\mathbb{R}^{n})=:H^{(s)}(\mathbb{R}^{n})$ и
$H^{\varphi}(\Gamma)=:H^{(s)}(\Gamma)$ есть (гильбертовы) пространства Соболева
порядка $s$, заданные на $\mathbb{R}^{n}$ и $\Gamma$ соответственно.

Класс пространств
$\{H^{\varphi}(\mathbb{R}^{n}\,\mbox{либо}\,\Gamma):\varphi\in\mathrm{RO}\}$ мы
называем расширенной соболевской шкалой на $\mathbb{R}^{n}$ либо на $\Gamma$.

Обозначим $\varrho(t):=t$ при $t\geq1$. Матричный дифференциальный оператор $A$
является ограниченным оператором
\begin{equation}\label{f2.2}
A:\,\bigoplus_{k=1}^{p}H^{\varphi\rho^{m_k}}(\Gamma)\rightarrow(H^{\varphi}(\Gamma))^{p}
\quad\mbox{для каждого}\;\;\varphi\in\mathrm{RO}
\end{equation}
(см. ниже лемму~\ref{lem4.1}). Заметим, что здесь $\varphi\rho^{m_k}\in\mathrm{RO}$;
следовательно, левое пространство в \eqref{f2.2} определено.

\section{Основные результаты}\label{Sec3}

Сформулируем результаты статьи о свойствах оператора \eqref{f2.2}; их доказательство
будет дано ниже в п.~\ref{Sec5}.

Обозначим через $A^{+}$ матричный дифференциальный оператор, формально сопряженный к
$A$ относительно формы $(\cdot,\cdot)_{\Gamma}$. Он определяется условием:
$(Au,v)_{\Gamma}=(u,A^{+}v)_{\Gamma}$ для любых вектор-функций
$u,v\in(C^\infty(\Gamma))^p$. Здесь и далее
$(f,v)_{\Gamma}:=(f_{1},v_{1})_{\Gamma}+\ldots+(f_{p},v_{p})_{\Gamma}$, если
$f=(f_{1},\ldots,f_{p})\in(\mathcal{D}'(\Gamma))^{p}$ и
$v=(v_{1},\ldots,v_{p})\in(C^\infty(\Gamma))^{p}$.

Положим
\begin{gather*}
N:=\{u\in(C^\infty(\Gamma))^p:\,Au=0\;\;\mbox{на}\;\;\Gamma\},\\
N^{+}:=\{v\in(C^\infty(\Gamma))^p:\,A^{+}v=0\;\;\mbox{на}\;\;\Gamma\}.
\end{gather*}
Поскольку системы $Au=f$ и $A^{+}v=g$ эллиптичны на $\Gamma$, то пространства $N$ и
$N^{+}$ конечномерны \cite[теорема 3.2.1]{Agranovich94}.

Напомним, что линейный ограниченный оператор $T:E_1\rightarrow E_2$, где $E_1$ и
$E_2$~--- банаховы пространства, называется фредгольмовым, если его ядро $\ker T$ и
коядро $\mathrm{coker}\,T:=E_2/T(X)$ конечномерны. У фредгольмового оператора $T$
область значений $T(X)$ замкнута в $E_2$, а индекс $\mathrm{ind}\,T:=\dim\ker
T-\dim\mathrm{coker}\,T$ конечен.

\begin{theorem}\label{th3.1}
Ограниченный оператор \eqref{f2.2} фредгольмов для любого параметра
$\varphi\in\mathrm{RO}$. Его ядро равно $N$, область значений совпадает с
подпространством
\begin{equation}\label{f3.1}
\bigl\{f\in(H^{\varphi}(\Gamma))^p:\,(f,v)_{\Gamma}=0\;\;\mbox{для всех}\;\;v\in
N^{+}\bigr\},
\end{equation}
а индекс равен $\dim N-\dim N^{+}$ и  не зависит от $\varphi$.
\end{theorem}

В случае, когда оба пространства $N$ и $N^{+}$ тривиальны, оператор \eqref{f2.2}
является изоморфизмом в силу теоремы Банаха об обратном операторе. В общей ситуации
изоморфизм удобно задавать с помощью следующих проекторов.

Пусть $\varphi\in\mathrm{RO}$. Разложим пространства, в которых действует
фредгольмов оператор \eqref{f2.2}, в прямые суммы подпространств:
\begin{equation*}
\bigoplus_{k=1}^{p}H^{\varphi\rho^{m_k}}(\Gamma)=N\dotplus
\biggl\{u\in\bigoplus_{k=1}^{p}H^{\varphi\rho^{m_k}}(\Gamma):
(u,w)_{\Gamma}=0\;\mbox{для всех}\;w\in N\biggr\},
\end{equation*}
\begin{equation*}
(H^{\varphi}(\Gamma))^p=N^{+}\dotplus\bigl\{f\in
(H^{\varphi}(\Gamma))^p:(f,v)_{\Gamma}=0\;\mbox{для всех}\;v\in N^{+}\bigr\}.
\end{equation*}
(Такие разложения существуют, поскольку в них слагаемые имеют тривиальное
пересечение, и конечная размерность первого слагаемого равна коразмерности второго.)

Обозначим через $P$ и $P^{+}$ соответственно (косые) проекторы пространств
$\bigoplus_{k=1}^p\,H^{\varphi\rho^{m_k}}(\Gamma)$ и $(H^{\varphi}(\Gamma))^p$ на
вторые слагаемые в указанных суммах параллельно первым слагаемым. Эти проекторы (как
отображения) не зависят от $\varphi$.

\begin{theorem}\label{th3.2}
Для каждого $\varphi\in\mathrm{RO}$ сужение оператора $\eqref{f2.2}$ на
подпространство $P\bigl(\bigoplus_{k=1}^p\,H^{\varphi\rho^{m_k}}(\Gamma)\bigr)$
является изоморфизмом
\begin{equation}\label{f3.2}
A:P\biggl(\,\bigoplus_{k=1}^p\,H^{\varphi\rho^{m_k}}(\Gamma)\biggr)\leftrightarrow
P^{+}\bigl((H^{\varphi}(\Gamma))^p\bigr).
\end{equation}
\end{theorem}

Следующая теорема содержит априорную оценку решения эллиптической системы $Au=f$ в
расширенной соболевской шкале.

\begin{theorem}\label{th3.3}
Пусть заданы параметры: функциональный $\varphi\in\mathrm{RO}$ и числовой
$\sigma>0$. Тогда существует число $c=c(\varphi,\sigma)>0$ такое, что для
произвольных вектор-функций
\begin{equation}\label{f3.3}
u \in \bigoplus_{k=1}^p\,H^{\varphi\rho^{m_k}}(\Gamma)\quad\mbox{и}\quad f\in
(H^{\varphi}(\Gamma))^p,
\end{equation}
удовлетворяющих уравнению $Au=f$ на $\Gamma$, справедлива оценка
\begin{equation}\label{f3.4}
\biggl(\,\sum_{k=1}^{p}\|u_{k}\|_{{\varphi\rho}^{m_k}}^2\biggr)^{1/2}\leq
c\,\biggl(\,\sum_{j=1}^{p}\|f_{j}\|_{\varphi}^2\biggr)^{1/2}+
c\,\biggl(\,\sum_{k=1}^{p}\|u_{k}\|_{{\varphi \rho}^{m_k-\sigma}}^2\biggr)^{1/2}.
\end{equation}
\end{theorem}

Если $N=\{0\}$, то в правой части оценки \eqref{f3.4} отсутствует последнее
слагаемое.

Исследуем локальную регулярность решения эллиптической системы $Au=f$. Пусть $V$~---
произвольное открытое непустое подмножество многообразия $\Gamma$. Как и прежде,
$\varphi\in\mathrm{RO}$. Обозначим через $H_{\mathrm{loc}}^{\varphi}(V)$ линейное
топологическое пространство всех распределений $h\in\mathcal{D}'(\Gamma)$ таких, что
$\chi h\in H^{\varphi}(\Gamma)$ для любой функции $\chi\in C^\infty(\Gamma)$, у
которой $\mathrm{supp}\,\chi\subset V$. Топология в $H_{\mathrm{loc}}^{\varphi}(V)$
задается посредством полунорм $h\mapsto\|\chi h\|_{\varphi}$, где $\chi$~---
произвольная указанная функция. Если $V=\Gamma$, то
$H_{\mathrm{loc}}^{\varphi}(V)=H^{\varphi}(\Gamma)$.

\begin{theorem}\label{th3.4}
Пусть $\varphi\in\mathrm{RO}$. Предположим, что вектор-функция
$u\in(\mathcal{D}'(\Gamma))^p$ является решением уравнения $Au=f$ на открытом
множестве $V\subseteq\Gamma$, где $f\in(H_{\mathrm{loc}}^{\varphi}(V))^p$. Тогда
\begin{equation}\label{f3.5}
u\in\bigoplus_{k=1}^{p}H_{\mathrm{loc}}^{\varphi\rho^{m_k}}(V).
\end{equation}
\end{theorem}

Следующая теорема дает (в терминах расширенной соболевской шкалы) достаточное
условие существование непрерывных производных у решения эллиптической системы
$Au=f$.

\begin{theorem}\label{th3.5}
Пусть заданы целые числа $k\in\{1,\ldots,p\}$ и $r\geq0$, а функциональный параметр
$\varphi\in\mathrm{RO}$ удовлетворяет условию
\begin{equation}\label{f3.6}
\int\limits_{1}^{\infty}\,t^{2r+n-1-2m_k}\,\varphi^{-2}(t)\,dt<\infty.
\end{equation}
Предположим, что вектор-функция $u\in(\mathcal{D}'(\Gamma))^p$ является решением
уравнения $Au=f$ на открытом множестве $V\subseteq\Gamma$, где
$f\in(H_{\mathrm{loc}}^{\varphi}(V))^p$. Тогда компонента $u_k$ решения имеет на
множестве $V$ непрерывные производные порядка $\leq r$, т.~е. $u_k\in \nobreak
C^r(V)$.
\end{theorem}

Из этой теоремы вытекает следующее достаточное условие классичности решения~$u$.

\begin{theorem}\label{th3.6}
Предположим, что вектор-функция $u\in(\mathcal{D}'(\Gamma))^p$ является решением
уравнения $Au=f$ на открытом множестве $V\subseteq\Gamma$, где
$f\in(H_{\mathrm{loc}}^{\varphi}(V))^p$, а функциональный параметр
$\varphi\in\mathrm{RO}$ удовлетворяет условию
\begin{equation*}
\int\limits_{1}^{\infty}\,t^{n-1}\,\varphi^{-2}(t)\,dt<\infty.
\end{equation*}
Тогда
\begin{equation}\label{f3.7}
u\in\bigoplus_{k=1}^{p}C^{m_k}(V),
\end{equation}
т.~е. решение $u$ является классическим на~$V$.
\end{theorem}

В связи с последним результатом заметим следующее. Если выполняется условие
\eqref{f3.7}, то функция $f=Au$ вычисляется (в локальных координатах) на $V$ с
помощью классических производных, поскольку на каждую компоненту $u_{k}\in
C^{m_k}(V)$ действуют дифференциальные операторы $A_{j,k}$ порядка $\leq m_k$.

В конце этого пункта отметим, что аналоги теорем \ref{th3.1}--\ref{th3.5} доказаны
ранее в \cite{08MFAT2} для более узкого класса пространств Хермандера~--- уточненной
соболевской шкалы~--- применительно к более широкому семейству эллиптических по
Дуглису--Ниренбергу систем. Для них последняя теорема~\ref{th3.6} не имеет аналога.
Заметим, что уточненная соболевская шкала (на $\Gamma$) образована пространствами
Хермандера $H^{\varphi}(\Gamma)$, для которых параметр $\varphi$ является правильно
меняющейся на бесконечности функцией произвольного вещественного порядка
\cite{BinghamGoldieTeugels89, Seneta76}. Рассматриваемый в настоящей работе класс
$\mathrm{RO}$ параметров $\varphi$ содержит как правильно меняющиеся функции, так и
иные (не имеющие порядка изменения на бесконечности).

\section{Вспомогательные факты и результаты}\label{Sec4}

Приведем некоторые полезные нам факты; они будут использованы в доказательствах
теорем.

Отметим следующие свойства функционального класса RO (см. цитированные выше
монографии \cite{BinghamGoldieTeugels89, Seneta76}).

Функция $\varphi$ принадлежит классу $\mathrm{RO}$ тогда и только тогда, когда
$$
\varphi(t)=\exp\Biggl(\beta(t)+
\int\limits_{1}^{\:t}\frac{\gamma(\tau)}{\tau}\;d\tau\Biggr)\quad \mbox{при}\quad
t\geq1,
$$
где вещественные функции $\beta$ и $\gamma$  измеримы по Борелю и ограничены на
полуоси $[1,\infty)$.

Для любой функции $\varphi\in\mathrm{RO}$ существуют числа
$s_{0},s_{1}\in\mathbb{R}$, $s_{0}\leq s_{1}$, и $c_{0},c_{1}>0$ такие, что
\begin{equation}\label{f4.1}
c_{0}\lambda^{s_{0}}\leq\frac{\varphi(\lambda t)}{\varphi (t)}\leq
c_{1}\lambda^{s_{1}} \quad\mbox{для всех}\quad t\geq1,\;\;\lambda\geq1.
\end{equation}

Положим
\begin{gather*}
\sigma_{0}(\varphi):=
\sup\,\{s_{0}\in\mathbb{R}:\,\mbox{верно левое неравенство в \eqref{f4.1}}\},\\
\sigma_{1}(\varphi):=\inf\,\{s_{1}\in\mathbb{R}:\,\mbox{верно правое неравенство в
\eqref{f4.1}}\}.
\end{gather*}
Числа $\sigma_{0}(\varphi)$ и $\sigma_{1}(\varphi)$ конечны и называются
соответственно нижним и верхним индексом Матушевской \cite{Matuszewska64} функции
$\varphi\in\mathrm{RO}$.

Далее рассмотрим ряд свойств расширенной соболевской шкалы на многообразии~$\Gamma$.

Пусть $\varphi,\varphi_{1}\in\mathrm{RO}$. Функция $\varphi(t)/\varphi_{1}(t)$
ограничена в окрестности $+\infty$ тогда и только тогда, когда
$H^{\varphi_{1}}(\Gamma)\hookrightarrow H^{\varphi}(\Gamma)$; это вложение
непрерывно и плотно. Оно компактно тогда и только тогда, когда
$\varphi(t)/\varphi_{1}(t)\rightarrow0$ при $t\rightarrow+\infty$.

В частности, выполняются компактные и плотные вложения
\begin{equation}\label{f4.2}
H^{(s_1)}(\Gamma)\hookrightarrow H^{\varphi}(\Gamma)\hookrightarrow
H^{(s_0)}(\Gamma)\;\;\mbox{для всех}\;\;
s_{1}>\sigma_{1}(\varphi),\;s_{0}<\sigma_{0}(\varphi).
\end{equation}

Пусть заданы целое число $r\geq0$ и функциональный параметр $\omega\in\mathrm{RO}$.
Тогда
\begin{equation}\label{f4.3}
\int\limits_{1}^{\infty}\,t^{2r+n-1}\,\omega^{-2}(t)\,dt<\infty\;\Leftrightarrow\;
H^{\omega}(\Gamma)\hookrightarrow C^{r}(\Gamma),
\end{equation}
причем вложение непрерывно. Здесь, как обычно, $C^{r}(\Gamma)$~--- пространство
Гельдера на $\Gamma$ порядка $r$.

Эти свойства вытекают из соответствующих свойств пространств Хермандера $B_{2,k}$
\cite[п. 2.2]{Hermander63} (сравнить с \cite[п. 2.1.2, 2.1.4]{MikhailetsMurach10}).

В связи с \eqref{f4.3} отметим следующее. Если
$H^{\omega}(\Gamma)=H^{(s)}(\Gamma)$~--- пространство Соболева порядка $s$, т.~е.
$\omega(t)=t^{s}$ при $t\geq1$, то левая часть формулы \eqref{f4.3} равносильна
неравенству $s>r+n/2$, и мы приходим к теореме вложения Соболева.

Фундаментальное (и необходимое нам) свойство расширенной соболевской шкалы состоит в
том, что каждое пространство $H^{\varphi}(\Gamma)$, где $\varphi\in\mathrm{RO}$,
есть результат интерполяции с подходящим функциональным параметром пары соболевских
пространств $H^{(s_0)}(\Gamma)$ и $H^{(s_1)}(\Gamma)$, фигурирующих в \eqref{f4.2}.
Напомним определение этой интерполяции в случае общих гильбертовых пространств и
некоторые ее свойства (cм., например, монографию \cite[п. 1.1]{MikhailetsMurach10}).
Для наших целей достаточно ограничиться сепарабельными пространствами.

Пусть задана упорядоченная пара $X:=[X_{0},X_{1}]$ сепарабельных комплексных
гильбертовых пространств $X_{0}$ и $X_{1}$ такая, что выполняется непрерывное и
плотное вложение $X_{1}\hookrightarrow X_{0}$. Пару $X$ называем допустимой. Для нее
существует изометрический изоморфизм $J:X_{1}\leftrightarrow X_{\,0}$ такой, что
$J$~--- самосопряженный положительно определенный оператор в пространстве $X_{0}$ с
областью определения $X_{1}$. Оператор $J$ определяется парой $X$ однозначно; он
называется порождающим для $X$.

Обозначим через $\mathcal{B}$ множество всех измеримых по Борелю функций
$\psi:(0,\infty)\rightarrow(0,\infty)$, которые отделены от нуля на каждом множестве
$[r,\infty)$ и ограниченны на каждом отрезке $[a,b]$, где $r>0$ и $0<a<b<\infty$.

Пусть $\psi\in\mathcal{B}$. В пространстве $X_{0}$ определен, как функция от $J$,
оператор $\psi(J)$. Обозначим через $[X_{0},X_{1}]_\psi$ или, короче, $X_{\psi}$
область определения оператора $\psi(J)$, наделенную скалярным произведением $(w_1,
w_2)_{X_\psi}:=(\psi(J)w_1,\psi(J)w_2)_{X_0}$ и соответствующей нормой
$\|w\|_{X_\psi} = (w,w)_{X_\psi}^{1/2}$. Пространство $X_\psi$ гильбертово и
сепарабельно, причем выполняется непрерывное и плотное вложение $X_\psi
\hookrightarrow X_0$.

Функцию $\psi\in\mathcal{B}$ называем интерполяционным параметром, если для
произвольных допустимых пар $X=[X_0, X_1]$, $Y=[Y_0, Y_1]$ гильбертовых пространств
и для любого линейного отображения $T$, заданного на $X_0$, выполняется следующее.
Если при каждом $j\in\{0,1\}$ сужение отображения $T$ на пространство $X_{j}$
является ограниченным оператором $T:X_{j}\rightarrow Y_{j}$, то и сужение
отображения $T$ на пространство $X_\psi$ является ограниченным оператором
$T:X_{\psi}\rightarrow Y_{\psi}$. Тогда будем говорить, что пространство $X_\psi$
получено интерполяцией с функциональным параметром $\psi$ пары $X$.

Функция $\psi\in\mathcal{B}$ является интерполяционным параметром тогда и только
тогда, когда она псевдовогнута в окрестности бесконечности, т.~е.
$\psi(t)\asymp\psi_{1}(t)$ при $t\gg1$ для некоторой положительной вогнутой функции
$\psi_{1}(t)$. Это вытекает из теоремы Ж.~Питре \cite{Peetre66, Peetre68} об
описании всех интерполяционных функций положительного порядка. (Как обычно,
$\psi\asymp\psi_{1}$ обозначает ограниченность обоих отношений $\psi/\psi_{1}$ и
$\psi_{1}/\psi$ на указанном множестве.)

Необходимое нам интерполяционное свойство расширенной соболевской шкалы~---
нижеследующее \cite[теорема 2.22]{MikhailetsMurach10}.

\begin{proposition}\label{prop4.1}
Пусть заданы функция $\varphi\in\mathrm{RO}$ и вещественные числа $s_0$, $s_1$
такие, что $s_0<\sigma_0(\varphi)$ и $s_1>\sigma_1(\varphi)$. Положим
\begin{equation}\label{f4.4}
\psi(t):=
\begin{cases}
\;t^{{-s_0}/{(s_1-s_0)}}\,
\varphi\bigl(t^{1/{(s_1-s_0)}}\bigr)&\text{при}\quad t\geq1, \\
\;\varphi(1)&\text{при}\quad0<t<1.
\end{cases}
\end{equation}
Тогда функция $\psi\in\mathcal{B}$ является интерполяционным параметром, и
\begin{equation}\label{f4.5}
[H^{(s_0)}(\Gamma),H^{(s_1)}(\Gamma)]_{\psi}=H^{\varphi}(\Gamma)
\end{equation}
с эквивалентностью норм.
\end{proposition}

Отметим \cite[теорема 2.24]{MikhailetsMurach10}, что расширенная соболевская шкала
на $\Gamma$ совпадает (с точностью до эквивалентности норм) с классом всех
гильбертовых пространств, интерполяционных для пар соболевских пространств
$[H^{(s_0)}(\Gamma),H^{(s_1)}(\Gamma)]$, где $s_0,s_1\in\mathbb{R}$ и $s_0<s_1$.
(Аналогичный факт верен и для расширенной соболевской шкалы на $\mathbb{R}^{n}$.)
Это следует из теоремы В.~И.~Овчинникова \cite[п.~11.4]{Ovchinnikov84} об описании
всех гильбертовых пространств, интерполяционных для произвольно заданной совместимой
пары гильбертовых пространств. Напомним, что свойство (гильбертового) пространства
$H$ быть интерполяционным для допустимой пары $X=[X_0, X_1]$ означает следующее: а)
выполняются непрерывные вложения $X_1\hookrightarrow H\hookrightarrow X_0$, б)
всякий линейный оператор, ограниченный на каждом из пространств $X_0$ и $X_1$,
является ограниченным и на $X$.

Нам понадобится следующий факт, вытекающий из предложения \ref{prop4.1}.

\begin{lemma}\label{lem4.1}
Пусть $L$~--- произвольный линейный дифференциальный оператор на $\Gamma$ порядка
$l$ с коэффициентами класса $C^{\infty}(\Gamma)$. Он является ограниченным
оператором
$$
L:\,H^{\varphi\varrho^{l}}(\Gamma)\rightarrow H^{\varphi}(\Gamma)\quad\mbox{для
любого}\;\;\varphi\in\mathrm{RO}.
$$
\end{lemma}

\begin{proof}
Пусть $\varphi\in\mathrm{RO}$, а числа $s_{0}$, $s_{1}$ и интерполяционный параметр
$\psi$ такие как в предложении \ref{prop4.1}. Рассмотрим ограниченные операторы
$$
L:\,H^{(s_{j}+l)}(\Gamma)\rightarrow H^{(s_{j})}(\Gamma)\quad\mbox{для}\quad j=0,1,
$$
действующие в соболевских пространствах. Применив интерполяцию с функциональным
параметром $\psi$, получим в силу предложения \ref{prop4.1} требуемый ограниченный
оператор
\begin{multline*}
L:\,H^{\varphi\varrho^{l}}(\Gamma)=
\bigl[H^{(s_0+l)}(\Gamma),H^{(s_1+l)}(\Gamma)\bigr]_\psi\\
\rightarrow\bigl[H^{(s_0)}(\Gamma),H^{(s_1)}(\Gamma)\bigr]_\psi=H^{\varphi}(\Gamma).
\end{multline*}
Заметим, что здесь первое равенство верно, поскольку
$s_{0}+l<\sigma_{0}(\varphi\rho^{l})$ и $s_{1}+l>\sigma_{1}(\varphi\rho^{l})$, а
функциональный параметр $\psi$ удовлетворяет соотношению $\eqref{f4.4}$, если в нем
заменить $s_0$ на $s_0+l$, $s_1$ на $s_1+l$ и $\varphi$ на $\varphi\rho^{l}$.
\end{proof}

При интерполяции наследуется не только ограниченность, но и фредгольмовость линейных
операторов при некоторых дополнительных условиях. Сформулируем этот результат
применительно к рассмотренному методу интерполяции \cite[теорема
1.7]{MikhailetsMurach10}.

\begin{proposition}\label{prop4.3}
Пусть $X=[X_0,X_1]$ и $Y=[Y_0,Y_1]$~--- допустимые пары гильбертовых пространств.
Пусть, кроме того, на $X_0$ задано линейное отображение $T$ такое, что его сужения
на пространства $X_j$, где $j=0,1$, являются ограниченными фредгольмовым операторами
$T:X_j\rightarrow Y_j$, имеющими общее ядро и одинаковый индекс. Тогда для
произвольного интерполяционного параметра $\psi\in\mathcal{B}$ ограниченный оператор
$T:X_\psi\rightarrow Y_\psi$ фредгольмов с теми же ядром и индексом, а его область
значений равна $Y_\psi\cap T(X_0)$.
\end{proposition}

\section{Доказательство основных результатов}\label{Sec5}

Докажем свойства эллиптической системы $Au=f$, сформулированные в п.~\ref{Sec3}.

\begin{proof}[Доказательство теоремы $\ref{th3.1}$]
В соболевском случае $\varphi=\varrho^{s}$, где произвольно выбрано
$s\in\mathbb{R}$, эта теорема известна (см., например, \cite[теорема
3.2.1]{Agranovich94} или \cite[теорема 8.69]{WlokaRowleyLawruk95}). Докажем ее для
любого $\varphi\in\mathrm{RO}$ с помощью интерполяции с функциональным параметром.

Выберем числа $s_{0}<\sigma_{0}(\varphi)$, $s_{1}>\sigma_{1}(\varphi)$ и рассмотрим
ограниченные фредгольмовы операторы
\begin{equation}\label{f5.1}
A:\,\bigoplus_{k=1}^{p}H^{(s_j+m_k)}(\Gamma)\rightarrow
\bigl(H^{(s_j)}(\Gamma)\bigr)^p\quad\mbox{для}\quad j=0,1,
\end{equation}
действующие в пространствах Соболева. Эти операторы имеют общее ядро $N$, одинаковый
индекс, равный $\dim N-\dim N^{+}$, и области значений
\begin{equation}\label{f5.2}
A\biggl(\,\bigoplus_{k=1}^{p}H^{(s_j+m_k)}(\Gamma)\biggr)
=\Bigl\{f\in\bigl(H^{(s_j)}(\Gamma)\bigr)^p:\,(f,v)_{\Gamma}=0\;\;\mbox{для
всех}\;\;v\in N^{+}\Bigr\}.
\end{equation}

Определим интерполяционный параметр $\psi$ по формуле $\eqref{f4.4}$. Согласно
предложению~\ref{prop4.3} фредгольмовость операторов \eqref{f5.1} влечет за собой
фредгольмовость ограниченного оператора
\begin{equation}\label{f5.3}
A:\, \biggl[\,\bigoplus_{k=1}^{p}H^{(s_0+m_k)}(\Gamma),\,
\bigoplus_{k=1}^{p}H^{(s_1+m_k)}(\Gamma)\biggr]_\psi
\rightarrow\bigl[\bigl(H^{(s_0)}(\Gamma)\bigr)^p,\bigl(H^{(s_1)}(\Gamma)
\bigr)^p\,\bigr]_\psi.
\end{equation}
Здесь в силу предложения \ref{prop4.1} получаем следующие равенства пространств с
эквивалентностью норм в них:
\begin{gather*}
\biggl[\,\bigoplus_{k=1}^{p}H^{(s_0+m_k)}(\Gamma),\,
\bigoplus_{k=1}^{p}H^{(s_1+m_k)}(\Gamma)\biggr]_\psi\\
=\,\bigoplus_{k=1}^{p}\bigl[H^{(s_0+m_k)}(\Gamma),H^{(s_1+m_k)}(\Gamma)\bigr]_\psi=
\,\bigoplus_{k=1}^{p}H^{\varphi\rho^{m_k}}(\Gamma),
\end{gather*}
\begin{equation*}
\bigl[\,\bigl(H^{(s_0)}(\Gamma)\bigr)^p,
\bigl(H^{(s_1)}(\Gamma)\bigr)^p\,\bigr]_\psi
=\bigl(\bigl[H^{(s_0)}(\Gamma),H^{(s_1)}(\Gamma)\bigr]_\psi\bigr)^p=
\bigl(H^\varphi(\Gamma)\bigr)^p.
\end{equation*}
(Левые равенства в этих формулах следуют из определения интерполяции; см., например,
\cite[теорема~1.5]{MikhailetsMurach10}.)

Таким образом, \eqref{f2.2}~--- это фредгольмов оператор \eqref{f5.3}, который в
силу предложения \ref{prop4.3} имеет те же ядро $N$ и индекс $\dim N-\dim N^{+}$,
что и операторы \eqref{f5.1}. Кроме того, область значений оператора \eqref{f2.2}
совпадает с пространством
$$
\bigl(H^\varphi(\Gamma)\bigr)^p\cap
A\biggl(\,\bigoplus_{k=1}^{p}H^{(s_0+m_k)}\biggr),
$$
равным \eqref{f3.1} ввиду \eqref{f5.2}.
\end{proof}

\begin{proof}[Доказательство теоремы $\ref{th3.2}$]
Пусть $\varphi\in\mathrm{RO}$. По теореме \ref{th3.1}, $N$ является ядром, а
$P^{+}((H^{\varphi}(\Gamma))^p)$~--- областью значений оператора \eqref{f2.2}.
Следовательно оператор \eqref{f3.2}~--- биекция. Кроме того, этот оператор
ограничен. Значит, он является изоморфизмом в силу теоремы Банаха об обратном
операторе.
\end{proof}

\begin{proof}[Доказательство теоремы $\ref{th3.3}$]
Обозначим для краткости
$$
\mathcal{X}:=\bigoplus_{k=1}^{p}H^{\varphi\rho^{m_k}}(\Gamma),\quad
\mathcal{Y}:=(H^{\varphi}(\Gamma))^p\quad\mbox{и}\quad
\mathcal{Z}:=\bigoplus_{k=1}^{p}H^{\varphi\rho^{m_k-\sigma}}(\Gamma).
$$
Пусть вектор-функции \eqref{f3.3} такие, что $Au=f$ на $\Gamma$. Для них требуемая
оценка \eqref{f3.4} следует из неравенств
$$
\|Pu\|_{\mathcal{X}}\leq c_{1}\|f\|_{\mathcal{Y}}\quad\mbox{и}\quad
\|(1-P)u\|_{\mathcal{X}}\leq c_{2}\|u\|_{\mathcal{Z}}.
$$
Здесь $c_1$~--- норма оператора, обратного к \eqref{f3.2} (заметим, что $APu=f)$, а
$c_2$~--- норма оператора $1-P:\mathcal{Z}\rightarrow N$, где $N$ рассматривается
как конечномерное подпространство в $\mathcal{X}$.
\end{proof}

\begin{proof}[Доказательство теоремы $\ref{th3.4}$]
Пространство $\mathcal{D}'(\Gamma)$ является объе\-динением соболевских пространств
$H^{(s)}(\Gamma)$, где $s\in\mathbb{R}$. Поэтому ввиду \eqref{f4.2} выполняется
включение
\begin{equation}\label{f5.4}
u\in\bigoplus_{k=1}^{p}H^{\varphi\rho^{m_k-\sigma}}(\Gamma)\quad\mbox{для
некоторого}\quad \sigma>0.
\end{equation}

Предварительно рассмотрим случай, когда $V=\Gamma$ (к нему сводится доказательство
теоремы в общей ситуации). В силу теоремы \ref{th3.1} имеем:
$$
(H^\varphi(\Gamma))^p \cap
A\biggl(\,\bigoplus_{k=1}^{p}H^{\varphi\rho^{m_k-\sigma}}(\Gamma)\biggr)=
A\biggl(\,\bigoplus_{k=1}^{p}H^{\varphi\rho^{m_k}}(\Gamma)\biggr).
$$
Поэтому, из условия $f \in (H^\varphi(\Gamma))^p$ и включения \eqref{f5.4} вытекает,
что
$$
f = Au \in A\biggl(\,\bigoplus_{k=1}^{p}H^{\varphi\rho^{m_k}}(\Gamma)\biggr).
$$
Таким образом, на $\Gamma$ наряду с равенством $Au=f$ выполняется также равенство
$Av=f$ для некоторого $v\in \bigoplus_{k=1}^{p}H^{\varphi\rho^{m_k}}(\Gamma)$.
Следовательно, $A(u-v)=0$ на $\Gamma$; отсюда $w:=u-v\in
N\subset(C^\infty(\Gamma))^p$ ввиду теоремы \ref{th3.1}. Поэтому $u=v+w \in
\bigoplus_{k=1}^{p}H^{\varphi\rho^{m_k}}(\Gamma)$, и теорема \ref{th3.4} доказана в
случае $V=\Gamma$.

Рассмотрим теперь общую ситуацию. Покажем, что из условия
$f\in(H^{\varphi}_{\mathrm{loc}}(V))^p$ вытекает следующее свойство повышения
регулярности решения уравнения $Au=f$ на $V$: для каждого числа $r \geq 1$
справедлива импликация
\begin{equation}\label{f5.5}
u\in\bigoplus_{k=1}^{p}H_{\mathrm{loc}}^{\varphi\rho^{m_k-r}}(V)\Rightarrow
u\in\bigoplus_{k=1}^{p}H_{\mathrm{loc}}^{\varphi\rho^{m_k -r+1}}(V).
\end{equation}

Произвольно выберем функцию $\chi\in C^{\infty}(\Gamma)$ такую, что
$\mathrm{supp}\,\chi\subset V$. Выберем также функцию $\eta\in C^{\infty}(\Gamma)$,
удовлетворяющую условиям: $\mathrm{supp}\,\eta\subset V$ и $\eta\equiv1$ в
окрестности $\mathrm{supp}\,\chi$. Переставив матричный дифференциальный оператор
$A$ и оператор умножения на функцию $\chi$, можем записать:
\begin{equation}\label{f5.6}
A(\chi u)=A(\chi\eta u)=\chi A(\eta u)+ A'(\eta u)=\chi f+A'(\eta u)\quad\mbox{на}
\quad\Gamma.
\end{equation}
Здесь $A'=(A'_{j,k})_{j,k=1}^p$~--- некоторый матричный дифференциальный оператор с
коэффициентами класса $C^{\infty}(\Gamma)$, удовлетворяющий условию
$\mathrm{ord}\,A'_{j,k}\leq \mathrm{ord}\,A_{j,k}-1$. Тогда
$\mathrm{ord}\,A'_{j,k}\leq m_k-1$ и ввиду леммы \ref{lem4.1} получаем:
\begin{equation}\label{f5.7}
u\in\bigoplus_{k=1}^{p}H_{\mathrm{loc}}^{\varphi\rho^{m_k-r}}(V) \Rightarrow A'\eta
u\in\bigl(H^{\varphi\rho^{-r+1}}(\Gamma)\bigr)^p.
\end{equation}
Кроме того, поскольку $f\in(H_{\mathrm{loc}}^{\varphi}(V))^p$ и $r\geq1$, то
\begin{equation}\label{f5.8}
\chi f \in (H^{\varphi}(\Gamma))^p
\subseteq\bigl(H^{\varphi\rho^{-r+1}}(\Gamma)\bigr)^p.
\end{equation}

На основании формул \eqref{f5.6}, \eqref{f5.7}, \eqref{f5.8} и настоящей теоремы,
уже доказанной в случае $V=\Gamma$, имеем:
\begin{equation*}
u\in\bigoplus_{k=1}^{p}H_{\mathrm{loc}}^{\varphi\rho^{m_k-r}}(V)
\Rightarrow A(\chi u)\in\bigl(H^{\varphi\rho^{-r+1}}(\Gamma)\bigr)^p
\Rightarrow \chi u \in \bigoplus_{k=1}^{p}H^{\varphi\rho^{-r+1}\rho^{m_k}}(\Gamma).
\end{equation*}
Тем самым доказано $\eqref{f5.5}$ ввиду произвольности указанного выбора
функции~$\chi$.

Теперь при помощи \eqref{f5.4} и \eqref{f5.5} легко вывести включение \eqref{f3.5}.
В~\eqref{f5.4} можно считать, что $\sigma$~--- натуральное число. Применяя
\eqref{f5.5} последовательно для значений $r=\sigma$, $r=\sigma-1$, ..., $r=1$,
выводим требуемое включение:
\begin{gather*}
u\in\bigoplus_{k=1}^{p}H^{\varphi\rho^{m_k-\sigma}}(\Gamma)\subset
\bigoplus_{k=1}^{p}H_{\mathrm{loc}}^{\varphi \rho ^{m_k-\sigma}}(V)\\
\Rightarrow u \in \bigoplus_{k=1}^{p}H_{\mathrm{loc}}^{\varphi \rho
^{m_k-\sigma+1}}(V) \Rightarrow \ldots\Rightarrow u \in
\bigoplus_{k=1}^{p}H_{\mathrm{loc}}^{\varphi \rho ^{m_k}}(V).
\end{gather*}
\end{proof}

\begin{proof}[Доказательство теоремы $\ref{th3.5}$]
Согласно теореме \ref{th3.4} верно включение $u_k\in
H_{\mathrm{loc}}^{\varphi\rho^{m_k}}(V)$. Выберем произвольно точку $x\in V$ и
функцию $\chi\in C^{\infty}(\Gamma)$ такую, что $\mathrm{supp}\,\chi\subset V$ и
$\chi\equiv1$ в некоторой окрестности точки $x$. Тогда в силу условия \eqref{f3.6} и
свойства \eqref{f4.3}, где берем $\omega(t):=\varphi(t)t^{m_{k}}$ при $t\geq1$,
получаем:
$$
\chi u_k\in H^{\varphi\rho^{m_k}}(\Gamma)\subset C^{r}(\Gamma).
$$
Отсюда ввиду произвольности выбора точки $x\in V$ следует требуемое включение
$u_k\in C^{r}(V)$.
\end{proof}

\medskip

\end{document}